\documentclass[a4paper]{amsart}

\usepackage[utf8]{inputenc}
\usepackage[T1]{fontenc}
\usepackage{lmodern, enumerate}
\usepackage{amssymb,amsxtra,amscd,amsmath}
\usepackage{amsthm}
\usepackage{amsfonts}
\usepackage{bbm}
\usepackage[all]{xy}
\usepackage{xcolor}
\usepackage{nicefrac,mathtools}
\usepackage{microtype}
\usepackage[pdftitle={Deformation via coactions},
 pdfauthor={Alcides Buss, Siegfried Echterhoff},
 pdfsubject={Mathematics}
]{hyperref}
\usepackage[lite]{amsrefs}
\newcommand*{\MRref}[2]{ \href{http://www.ams.org/mathscinet-getitem?mr=#1}{MR #1}}
\newcommand*{\arxiv}[1]{\href{http://www.arxiv.org/abs/#1}{arXiv: #1}}
\usepackage{comment}

\numberwithin{equation}{section}
\theoremstyle{plain}
\newtheorem{theorem}[equation]{Theorem}
\newtheorem{lemma}[equation]{Lemma}
\newtheorem{proposition}[equation]{Proposition}

\theoremstyle{definition}

\newtheorem{notation}[equation]{Notation}
\theoremstyle{remark}
\newtheorem{remark}[equation]{Remark}

\DeclareMathOperator{\spn}{{span}}

\DeclareMathOperator{\id}{\mathrm{id}}

\newcommand*{\tig}{{\tilde{g}}}

\newcommand*{\nb}{\nobreakdash}
\newcommand*{\Star}{\(^*\)\nobreakdash-}

\newcommand*{\T}{\mathbb T}

\newcommand*{\C}{\mathbb C}

\newcommand*{\Lb}{\mathcal L}
\renewcommand*{\L}{\mathcal L}
\newcommand*{\K}{\mathcal K}
\renewcommand*{\H}{\mathcal H}

\newcommand*{\cont}{C}
\newcommand*{\contc}{\cont_c}
\newcommand*{\M}{\mathcal M}

\newcommand*{\Ad}{\textup{Ad}}

\newcommand*{\U}{\mathcal U}
\newcommand*{\E}{\mathcal E}
\newcommand*{\X}{\mathcal X}
\newcommand{\rt}{\mathrm{rt}}
\newcommand{\Rt}{\mathrm{Rt}}

\newcommand*{\congto}{\xrightarrow\sim}

\newcommand*{\cstar}{\texorpdfstring{$C^*$\nobreakdash-\hspace{0pt}}{*-}}
\newcommand*{\into}{\hookrightarrow}
\newcommand*{\onto}{\twoheadrightarrow}

\renewcommand*{\max}{\mathrm{max}}

\newcommand*{\dual}[1]{\widehat{#1}}
\newcommand*{\dualG}{\widehat{G}}

\newcommand{\om}{\omega}

\newcommand{\car}{\curvearrowright}

\newcommand{\B}{\mathcal B}

\newcommand{\hatG}{\widehat{G}}
\newcommand{\hatdelta}{\widehat{\delta}}

\begin{document}
\title[Fischer's approach to deformation]{Fischer's approach to deformation of coactions}

\author{Alcides Buss}
\email{alcides.buss@ufsc.br}
\address{Departamento de Matem\'atica\\
 Universidade Federal de Santa Catarina\\
 88.040-900 Florian\'opolis-SC\\
 Brazil}

\author{Siegfried Echterhoff}
\email{echters@uni-muenster.de}
\address{Mathematisches Institut\\
Universit\"at M\"un\-ster\\
 Einsteinstr.\ 62\\
 48149 M\"unster\\
 Germany}

\begin{abstract}
This paper explores a novel approach to the deformation of \cstar{}algebras via coactions of locally compact groups, emphasizing Fischer's construction in the context of maximal coactions. We establish a rigorous framework for understanding how deformations arise from group coactions, extending previous work by Bhowmick, Neshveyev, and Sangha. Using Landstad duality, we compare different deformation procedures, demonstrating their equivalence and efficiency in constructing twisted versions of given \cstar{}algebras. Our results provide deeper insights into the interplay between exotic crossed products, coaction duality, and operator algebra deformations, offering a unified perspective for further generalizations.
\end{abstract}

\subjclass[2010]{46L55, 22D35}

\keywords{Deformation, Borel cocycle, coactions, Fischer maximalization, Landstad duality, exotic crossed products}

\thanks{This work was funded by: the Deutsche Forschungsgemeinschaft (DFG, German Research Foundation) Project-ID 427320536 SFB 1442 and under Germany's Excellence Strategy EXC 2044  390685587, Mathematics M\"{u}nster: Dynamics, Geometry, Structure; and CNPq/CAPES/Humboldt and Fapesc - Brazil.}

\maketitle


\section{introduction}\label{sec-intro}
Motivated by earlier results of several previous works (e.g., \cites{Rieffel:Deformation, Kasprzak1, BNS}), the authors introduced in \cite{BE:deformation} a general procedure for deformation of $C^*$-algebras via coactions of groups on $C^*$-algebras. 
A key ingredient for our deformation procedure is a version of Landstad duality for (possibly exotic) crossed products by coactions. 
Specifically, for a locally compact group $G$, we let $\rt:G\car C_0(G)$ denote the right translation action. A weak $G\rtimes G$ algebra is a $C^*$-algebra $B$ equipped with an action $\beta:G\car B$ 
and an $\rt-\beta$-equivariant nondegenerate $C^*$-homomorphism $\phi:C_0(G)\to \M(B)$. Assume that $\rtimes_\mu$ is an exotic crossed-product functor which admits dual coactions. 
Explicitly, we consider a functorial crossed-product construction $(B,G,\beta)\mapsto B\rtimes_{\beta,\mu}G$ such that $B\rtimes_{\beta,\mu}G$ is a $C^*$-completion of the usual convolution algebra $C_c(G,B)$
with a $C^*$-norm $\|\cdot\|_\mu$ satisfying $\|\cdot\|_r\leq \|\cdot\|_\mu\leq \|\cdot\|_{\max}$. Further, we assume that the dual coaction 
$$\widehat{\beta}_{\max}: B\rtimes_{\beta,\max}G\to \M(B\rtimes_{\beta,\max}G\otimes C^*(G))$$ 
factors through a coaction $\widehat{\beta}_\mu$ on $B\rtimes_{\beta,\mu}G$. 

The results of \cite{Buss-Echterhoff:Exotic_GFPA} on (exotic) generalized fixed-point algebras, then yield the construction of a unique (up to isomorphism) cosystem $(A_\mu,\delta_\mu)$ such that 
$$(B,G,\beta)\cong (A_\mu\rtimes_{\delta_\mu}\widehat{G}, G, \widehat{\delta_\mu})$$ 
and $(A_\mu,\delta_\mu)$ satisfies $\mu$-Katayama duality in the sense that 
$$B\rtimes_{\beta,\mu}G\cong A_\mu\rtimes_{\delta_\mu}\widehat{G}\rtimes_{\widehat{\delta_\mu},\mu}G\cong A_\mu\otimes \K(L^2(G)).$$
Starting from any coaction $(A,\delta)$ of $G$, the crossed product 
$B:=A\rtimes_\delta G$, together with the dual action $\beta:=\widehat{\delta}$ and the canonical inclusion $\phi:=j_{C_0(G)}:C_0(G)\to \M(A\rtimes_\delta G)$, forms a weak $G\rtimes G$-algebra. We recover $(A,\delta)$ via the above procedure if and only if $(A,\delta)$ satisfies $\mu$-Katayama duality. 

Applying the procedure to the maximal crossed-product functor $\rtimes_{\max}$ yields the {\em maximalization} $(A_{\max}, \delta_{\max})$ of $(A,\delta)$.  Using the reduced crossed-product functor $\rtimes_r$ provides the {\em normalization} (or reduction) of $(A,\delta)$. We refer to \cite{BE:deformation} for a concise survey of these constructions and facts.

Furthermore, considering a central extension $\sigma=(\T\stackrel{\iota}\into  G_\sigma \stackrel{q}\onto G)$, referred to as a {\em twist} for $G$, the authors constructed in \cite{BE:deformation} a 
{\em $\sigma$-deformed} weak $G\rtimes G$-algebra $(B^\sigma, \beta^\sigma, \phi^\sigma)$ out of the given triple $(B, \beta, \sigma)$. Landstad duality then produces a deformed cosystem $(A^\sigma_\mu, \delta^\sigma_\mu)$ associated with $(A_\mu,\delta_\mu)$. Our motivation was to extend -- and somehow simplify -- the construction due to Bhowmick, Neshveyev, and Sangha \cite{BNS}, initially formulated for deformation of coactions by group cocycles $\om\in Z^2(G,\T)$, which was restricted to normal coactions and reduced crossed products.

The relationship between group twists $\sigma$ and 
Borel group cocycles $\om \in Z^2(G,\T)$ is governed by the well-understood classification of central extensions of $G$ by $\T$ through the second Borel cohomology group $H^2(G,\T)$. In \cite{BE:deformation} the authors claimed, without proof, that their construction coincides with that in \cite{BNS} for normal coactions. One of the main goals of this paper is to provide a rigorous proof of this claim.

To this end, we adopt an alternative approach to Landstad duality, which is modeled after the construction of the maximalization of a coaction due to Fischer (\cite{Fischer-PhD}*{\S 4.5}) in the general case of regular locally compact quantum groups.  In the specific case of groups and maximal crossed products, a very detailed account is given in \cites{KOQ, KLQ:R-coactions}. The basic idea is to recover 
 a $C^*$-algebra $A$ from the tensor product $A\otimes \K$ by some algebra $\K=\K(\H)$ of compact operators
  as the set of all elements $a\in \M(A\otimes \K)$ such that for all $k\in \K$ we have
$a(1\otimes k)=(1\otimes k)a\in A\otimes \K$.

  In general,  an abstract $C^*$-algebra $E$ is isomorphic to a tensor product $A\otimes \K$ if and only if there exists a nondegenerate 
  $*$-homomorphism $i_\K:\K\to \M(E)$. We then have 
  \begin{equation*}
  A=C(E,i_\K):=\{a\in \M(E): a i_\K(k)=i_\K(k)a\in E\;\forall k\in \K\}
  \end{equation*}
  with isomorphism $A\otimes \K\cong E$ given by $a\otimes k\mapsto ai_\K(k)\in E$ for $a\in A, k\in \K$. If $\epsilon:E\to \M(E\otimes C^*(G))$ is a coaction of $G$ on $E$ which trivializes the image $i_\K(\K)\subseteq \M(E)$
in the sense that $\epsilon(i_\K(k))=i_{\K}(k)\otimes 1$ for all $k\in \K$, it has been shown by Fischer in \cite{Fischer-PhD} (but see \cites{KOQ, KLQ:R-coactions} for a detailed elaboration in the group case) that $\epsilon$ restricts to a coaction $\delta:A\to \M(A\otimes C^*(G))$.

Given a weak $G\rtimes G$-algebra $(B,\beta,\phi)$ as above and any duality crossed-product functor $\rtimes_{\mu}$ for $G$, the descent 
$$i_\K:=\phi\rtimes_\mu G: C_0(G)\rtimes_{\rt,\mu}G\cong \K(L^2(G))\to \M(B\rtimes_{\beta,\mu}G)$$
provides a $*$-homomorphism $i_{\K}:\K(L^2(G))\to \M(B\rtimes_{\beta,\mu}G)$
that is equivariant for the dual coactions $\widehat{\rt}$ and $\widehat{\beta_\mu}$. Using the canonical isomorphism $\K:=\K(L^2(G))\cong C_0(G)\rtimes_{\rt,\mu}G$ and applying a certain canonical exterior equivalence yields a coaction $\tilde{\beta}$ of
$G$ on $B\rtimes_{\beta,\mu}G$ which is trivial on  $i_\K(\K(L^2(G)))$. The 
restriction $\delta_\mu$ of $\tilde\beta$ to 
$A_\mu:=C(B\rtimes_{\beta,\mu}G,i_\K)$ provides us with the alternative description of the $\mu$-coaction $(A_\mu,\delta_\mu)$. 

For a twist $\sigma=(\T\stackrel{\iota}\into  G_\sigma \stackrel{q}\onto G)$ as above, instead of using a deformed weak $G\rtimes G$-algebra $(B^\sigma,\beta^\sigma, \phi^\sigma)$ as in \cite{BE:deformation}, we replace the crossed product $B\rtimes_{\beta,\mu}G$ 
in Fischer's construction by the twisted crossed product 
$B\rtimes_{(\beta,\iota^\sigma),\mu}G$. The structure map $\phi:C_0(G)\to \M(B)$ descents to an inclusion of twisted crossed products
$$\phi\rtimes G: C_0(G)\rtimes_{(\rt,\iota^\sigma)}G\cong \K(L^2(G))\to 
\M(B\rtimes_{(\beta,\iota^\sigma),\mu}G).$$
Again, replacing the dual coaction $\widehat{(\beta,\iota^\sigma)}$ by a suitable exterior equivalent coaction yields a coaction of $G$ on $B\rtimes_{(\beta,\iota^\sigma),\mu}G$ which trivializes 
the image of $\K(L^2(G))$ in $\M(B\rtimes_{(\beta,\iota^\sigma),\mu}G)$. Thus, Fischer's general procedure provides a cosystem $(A_\mu^\sigma, \delta_\mu^\sigma)$. However, it is not instantly clear that this 
$\sigma$-deformed cosystem is isomorphic to the one constructed previously. The major part of this paper is devoted to establish such an isomorphism. We achieve this by showing that there exists an isomorphism $$B^\sigma\rtimes_{\beta^\sigma, \mu}G\cong B\rtimes_{(\beta,\iota^\sigma),\mu}G$$ 
which is equivariant for the dual coactions and intertwines the inclusions of $\K(L^2(G))$. 
This leads directly to the canonical isomorphisms:
$$A_\mu^\sigma\rtimes_{\delta_\mu^\sigma}\dualG\rtimes_{\dual\delta_\mu^\sigma,\mu}G\cong A_\mu^\sigma\otimes\K\cong A_\mu^\sigma\rtimes_{\delta_\mu^\sigma}\dualG\rtimes_{(\dual\delta_\mu^\sigma,\iota^\sigma),\mu}G.$$
In particular, applying this to reduced crossed products, our results show that the reduced deformed pair $(A^\sigma_r, \delta^\sigma_r)$ aligns with the construction given in \cite{BNS}.

\section{Actions, coactions and their (exotic) crossed products}\label{sec-prel}

For terminology and notation concerning (co)actions, their (exotic) crossed products, and duality -- particularly Landstad duality for coactions in terms of generalized fixed-point algebras -- we refer the reader to our previous paper \cite{BE:deformation}. Here, we briefly recall some fundamental concepts and notation essential for the developments in this work.

Throughout the paper, $G$ denotes a locally compact group with a fixed Haar measure. Continuous actions of $G$ on a \cstar{}algebra $B$ will be written as $\beta\colon G\car B$. For simplicity, we denote the maximal crossed product by $B\rtimes_\beta G$, and $B\rtimes_{\beta,r}G$ for the reduced crossed product. In general, we write $B\rtimes_{\beta,\mu}G$ for any other (exotic) crossed product, meaning a \cstar completion of the convolution $*$-algebra $\contc(G,B)$ lying between the maximal and the reduced crossed product. 
 
A crossed product $B\rtimes_{\beta,\mu}G$ is called a {\em duality crossed product} if the dual coaction $\widehat{\beta}$ on $B\rtimes_{\beta}G$ factors through a coaction $\widehat{\beta}_\mu$ on $B\rtimes_{\beta,\mu}G$.
A {\em crossed-product functor} is a functor $(B,\beta)\mapsto B\rtimes_{\beta,\mu}G$ from the category of $G$-\cstar{}algebras (i.e. \cstar{}algebras endowed with continuous $G$-actions) to the category of 
\cstar{}algebras that sends actions $\beta:G\car B$ to crossed products $B\rtimes_{\beta,\mu}G$ such that for any $G$\nb{}-equivariant \Star{}homomorphism $\Phi:(B,\beta)\to (B',\beta')$ the associated \Star{}homomorphism $\Phi\rtimes_\mu G:B\rtimes_{\beta,\mu}G\to B'\rtimes_{\beta',\mu}G$
extends $\Phi\rtimes_{alg}G: C_c(G,B)\to C_c(G,B'); f\mapsto \Phi\circ f$. 
If all $\rtimes_\mu$-crossed products are duality crossed products, then $\rtimes_\mu$ is called a 
{\em duality crossed-product functor}. Duality functors exist in abundance, for example, the maximal and reduced crossed-product functors, as well as all correspondence functors as studied in \cite{BEW} are duality functors (\cite{BEW2}*{Theorem 4.14}).

A coaction of $G$ on a \cstar{}algebra $A$  will usually be denoted by the symbol $\delta\colon A\to \M(A\otimes C^*(G))$. Its crossed product will be written as $A\rtimes_\delta\dualG$. Recall that $A\rtimes_\delta\dualG$ can be 
realized as 
$$\overline{\spn}\big((\id\otimes\lambda)\circ\delta(A)(1\otimes M(C_0(G))\big)\subseteq \M(A\otimes \K(L^2(G))),$$
where $M:C_0(G)\to \B(L^2(G))$ is the representation by multiplication operators. We often write: 
$$j_A:=(\id\otimes\lambda)\circ \delta:A\to \M(A\rtimes_\delta\dualG)\;\text{and}\; j_{C_0(G)}:=1\otimes M:C_0(G)\to \M(A\rtimes_\delta\dualG)$$
for the canonical morphisms from $A$ and $C_0(G)$ into $\M(A\rtimes_\delta\dualG)$. The dual action $\widehat{\delta}:G\car A\rtimes_{\delta}\dualG$ is determined by the equation 
$$\widehat\delta_g\big(j_A(a)j_{C_0(G)}(f)\big)=j_A(a)j_{C_0(G)}(\rt_g(f)),$$
where $\rt:G\car C_0(G)$ denotes the action by right translations. 

Nilsen \cite{Nilsen:Duality}*{Corollary~2.6} showed that for every coaction $\delta:A\to\M(A\otimes C^*(G))$, there exists a canonical 
surjective \Star{}homomorphism
$$\Psi_{\max}: A\rtimes_\delta\dualG\rtimes_{\widehat\delta}G\onto A\otimes \K(L^2(G))$$
given as the integrated form of the covariant representation $(j_A\rtimes j_{C_0(G)}, 1\otimes\rho)$.
The coaction  $\delta$ is called {\em maximal} if $\Psi_\max$ is an isomorphism, and it is called {\em normal} if it factors through an isomorphism $ A\rtimes_\delta\dualG\rtimes_{\widehat\delta,r}G\congto A\otimes \K(L^2(G))$. In general, it factors through an isomorphism
\begin{equation}\label{eq:Katayama-mu-duality} \Psi_\mu: A\rtimes_\delta\dualG\rtimes_{\widehat\delta,\mu}G\congto A\otimes \K(L^2(G))
\end{equation}
for some (possibly exotic) duality crossed product $\rtimes_\mu$.\footnote{Note that $\rtimes_\mu$ may not always be associated to a crossed-product functor.} In this case, we say that $(A,\delta)$ is a $\mu$-coaction to indicate that it satisfies Katayama duality for the $\mu$-crossed product.

We write $\K:=\K(L^2(G))$ and define the coaction $\delta\otimes_* \id_{\K}:A\otimes \K\to \M(A\otimes \K\otimes C^*(G))$ by 
\begin{equation*}
\delta\otimes_* \id_{\K}:=(\id_A\otimes \Sigma)\circ (\delta\otimes \id_{\K}),
\end{equation*}
where $\Sigma: C^*(G)\otimes \K\to \K\otimes C^*(G)$ denotes the flip map. 
Let 
\begin{equation}\label{eq-wG}
w_G:=(s\mapsto u_s)\in U\M(C_0(G)\otimes C^*(G))\end{equation}
where $s\mapsto u_s\in U\M(C^*(G))$ denotes the canonical representation.
It is well known (e.g., see \cite{EKQ}*{Lemma 3.6}) that if 
$W:=( M \otimes \id_G)(w_G)\in U\M(\K(L^2(G))\otimes C^*(G))$, then $1\otimes W$ is a one-cocycle for 
$\delta\otimes_* \id_{\K}$. This leads to the new coaction 
\begin{equation}\label{eq-tildedelta}
\widetilde{\delta}:=\Ad (1\otimes W)\circ (\delta\otimes_* \id_{\K})
\end{equation}
of $G$ on $A\otimes \K$. The following proposition establishes a fundamental correspondence between coactions and their double duals:

\begin{proposition}\label{prop-coact-doubledual}
Suppose that $\|\cdot\|_\mu$ and $A\rtimes_{\delta}\widehat{G}\rtimes_{\widehat\delta,\mu}G$ are as above. Then the double dual coaction 
$$\widehat{\widehat{\delta}\,}:A\rtimes_\delta \widehat{G}\rtimes_{\widehat{\delta}}G\to \M(A\rtimes_\delta \widehat{G}\rtimes_{\widehat{\delta}}G\otimes C^*(G))$$
factors through a (double dual) coaction 
$\widehat{\widehat{\delta}\,}_{\!\!\mu}$ of $G$ on $A\rtimes_{\delta}\widehat{G}\rtimes_{\widehat\delta,\mu}G$ which corresponds to the coaction 
$\widetilde{\delta}$ on $A\otimes \K$ via the isomorphism $\Psi_\mu$.
\end{proposition}
\begin{proof} This is a direct consequence of 
  \cite{EKQ}*{Lemma 3.8}, which states that  the surjection $\Psi_\mu$ of ~\eqref{eq:Katayama-mu-duality} is $\widehat{\widehat\delta}-\tilde\delta$ 
equivariant. 
\end{proof}

The triple $(A\rtimes_\delta\dualG,\widehat{\delta}, j_{C_0(G)})$ appearing above is the prototype of what  we call a \emph{weak} $G\rtimes G$-algebra
$(B,\beta,\phi)$ as explaind in the introduction.
As a variant of the classical Landstad duality for reduced coactions \cite{Quigg:Landstad}, it is shown in \cite{Buss-Echterhoff:Exotic_GFPA} that,
given a weak $G\rtimes G$-algebra $(B,\beta,\phi)$, for any given \emph{duality crossed product} $B\rtimes_{\beta,\mu}G$, there exists a unique (up to isomorphism) $\mu$-coaction $(A_\mu, \delta_\mu)$ of $G$ such that  
\begin{equation}\label{eq-Landstad}
(A_\mu\rtimes_{\delta_\mu}\dual G, \widehat{\delta}_\mu, j_{C_0(G)})\cong (B, \beta, \phi).
\end{equation}
In particular, if we consider the maximal crossed-product functor $\rtimes_{\max}$, we
obtain the {\em maximalization} $(A_{\max}, \delta_{\max})$ of $(A,\delta)$ and if we use the reduced 
crossed-product functor $\rtimes_r$, we  will recover the {\em normalization} $(A_r,
\delta_r)$ of $(A,\delta)$. As a consequence of this, we obtain the following useful observation:

\begin{proposition}\label{prop-max-to-mu}
Suppose that  $\delta:A\to \M(A\otimes C^*(G))$ is a maximal coaction. Then, for every duality crossed-product  $A\rtimes_\delta\hatG\rtimes_{\hatdelta,\mu}G$, there exists a unique quotient $A_\mu$ of $A$ such that  
$\delta$ factors through a $\mu$-coaction 
$\delta_\mu:A_\mu\to\M(A_\mu\otimes C^*(G))$ and such that the canonical induced map
$A\rtimes_\delta \hatG\onto A_\mu\rtimes_{\delta_\mu}\hatG$ is an isomorphism.
\end{proposition}

We call $(A_\mu,\delta_\mu)$ the {\em $\mu$-ization } of $(A,\delta)$.
The  proposition implies that, in many cases, it suffices to focus on maximal coactions or the maximalization $(A_\max, \delta_\max)$ of a given coaction $(A,\delta)$, as the corresponding $\mu$-coactions $(A_\mu,\delta_\mu)$ can be recovered through the procedure outlined above.


\section{The Fischer approach to Landstad duality}\label{subsec-Fischer}
Landstad duality, as described in \S \ref{sec-prel}, provides the main tool for the \emph{deformation of $C^*$-algebras  by coactions}, as introduced in \cite{BE:deformation}. In this section, we present an alternative approach to Landstad duality based on Fischer's approach to maximalizations of coactions for regular locally compact quantum groups  (see \cite{Fischer-PhD}*{\S 4.5}). This approach aligns more closely with the constructions in \cite{BNS} and offers a more suitable perspective for possible generalizations to 
(regular) locally compact quantum groups.

In the group case, Fischer's approach towards the maximalization $(A_\max,\delta_\max)$ of a coaction $(A,\delta)$ has been studied in more detail in \cites{KOQ, KLQ:R-coactions}. Recall from the introduction that, given a nondegenerate $*$-homomorphism $i_{\K}:\K(\H)\to \M(E)$ for a $C^*$-algebra $E$, where $\H$ is a fixed Hilbert space, we obtain a canonical isomorphism 
$$E\cong A\otimes \K(\H),$$ 
where $A$ is defined as: 
\begin{equation}\label{eq-relcom}
  A=C(E,i_\K):=\{a\in \M(E): a i_\K(k)=i_\K(k)a\in E\;\forall k\in \K(\H)\}.
\end{equation}
The induced isomorphism $A\otimes \K(\H)\cong E$ is given by $a\otimes k\mapsto a i_\K(k)\in E$ for $a\in A, k\in \K(\H)$. This isomorphism clearly intertwines $i_\K:\K\to \M(E)$ with the canonical inclusion $\K\to \M(A\otimes \K); k\mapsto 1\otimes k$. The following result is due to Fischer \cite{Fischer-PhD}*{\S 4.5}. A detailed account for groups is also given in  \cite{KOQ}*{Lemma 3.2}.

  \begin{proposition}\label{prop-coact-Fischer}
  Suppose that $i_\K:\K\to \M(E)$ is a nondegenerate $*$-homomorphism and that $\epsilon: E\to\M(E\otimes C^*(G))$ is a coaction such that $\epsilon(i_\K(k))=i_\K(k)\otimes 1$ for all $k\in \K$ (i.e., $\epsilon$ is trivial on $i_\K(\K)$). Then $\epsilon$ restricts to a coaction $\delta: A\to \M(A\otimes C^*(G))$ with $A=C(E,i_\K)$ as in (\ref{eq-relcom}), such that 
$\epsilon$ corresponds under the isomorphism $E\cong A\otimes\K$ to the coaction $\delta\otimes_*\id_{\K}$ defined by
\begin{equation*}
\delta\otimes_* \id_{\K}:=(\id_A\otimes \sigma)\circ (\delta\otimes \id_{\K})
\end{equation*}
where $\sigma: C^*(G)\otimes \K\to \K\otimes C^*(G)$ denotes the flip map.
\end{proposition}

  \begin{remark}\label{rem-Fischer-functorial}
Fischer's construction is \emph{functorial} in the following sense: Suppose that $i_{\K(\H)}:\K(\H)\to \M(E)$ and 
$i_{\K(\H')}:\K(\H')\to \M(E')$ are nondegenerate $*$-homomorphisms and $\epsilon$ and $\epsilon'$ are coactions of $G$ on $E$ and $E'$ that are trivial on the images of
$i_{\K(\H)}$ and $i_{\K(\H)}$, respectively.
Let $(A, \delta)$ and $(A',\delta')$ denote the corresponding coactions as in the Proposition~\ref{prop-coact-Fischer}. 
If $\Phi:E\to E'$ is an $\epsilon-\epsilon'$ equivariant 
$*$-homomorphism such that $\Phi\big((i_{\K(\H)}(\K(\H))\big)=
i_{\K(\H')}(\K(\H'))$, then the restriction of $\Phi$ to $A\subseteq \M(E)$ induces a $\delta-\delta'$ equivariant homomorphism $\Phi|_A:A\to A'$.

In particular, if $\Phi:E\congto E'$ is an isomorphism, then $(A,\delta)\cong (A',\delta')$. The proof is given in \cite{Fischer-PhD}*{Anhang A} and also in 
\cite{KOQ}, where, indeed,  much more general functoriality properties of Fischer's construction are shown.
\end{remark}

Fischer’s approach allows us to obtain an alternative description of the coaction $(A_\mu,\delta_\mu)$, and hence to coaction Landstad duality: 
Recall that the crossed product $C_0(G)\rtimes_\rt G$
is isomorphic to $\K:=\K(L^2(G))$ via the covariant homomorphism $(M,\rho)$ (e.g. see \cite{Rieffel-Heisenberg}). Therefore, if $(B,\beta,\phi)$ is a weak $G\rtimes G$-algebra, then the $\rt-\beta$-equivariant $*$-homomorphism $\phi:C_0(G)\to \M(B)$ descents to a nondegenerate $*$-homomorphism 
$i_\K:\K\to \M(B\rtimes_\mu G)$ via the composition
$$\K\cong C_0(G)\rtimes_\rt G\stackrel{\phi\rtimes G}{\longrightarrow} \M(B\rtimes_{\max} G)\onto \M(B\rtimes_{\mu}G).$$
Now, if $(A_\mu,\delta_\mu)$ is the $\mu$-coaction corresponding to $(B,\beta,\phi)$ as in 
(\ref{eq-Landstad}), the isomorphism 
\begin{equation}\label{eq-isom}
B\rtimes_\mu G\cong A_\mu\rtimes_{\delta_\mu}\widehat{G}\rtimes_{\widehat{\delta_\mu},\mu}G\stackrel{\Psi_\mu}{\cong} 
A_\mu\otimes \mathcal K
\end{equation}
sends the image $i_\K(\K)\subseteq \M(B\rtimes_\mu G)$ to 
$1\otimes (M\rtimes \rho)(C_0(G)\rtimes_\rt G)=1\otimes \K$. Hence it follows that 
$A_\mu$ can be identified with the subalgebra 
 $$C(B\rtimes_\mu G, i_\K)=\{m\in \M(B\rtimes_\mu G): m i_\K(k)=i_\K(k)m\in B\rtimes_{\mu}G\; \forall k\in \K\}.$$ 
Moreover, if $B\rtimes_\mu G$ is a duality crossed product,  define 
\begin{equation}\label{eq-WB}
W_B:=((i_B\circ \phi)\otimes \id_G)(w_G)\in U\M(B\rtimes_\mu G\otimes C^*(G)),\end{equation}
 where 
$i_B\circ \phi:C_0(G)\to \M(B\rtimes_\mu G)$ is the composition of $\phi:C_0(G)\to \M(B)$ with the canonical homomorphism $i_B:B\to \M(B\rtimes_\mu G)$ and $w_G\in U\M(C_0(G)\otimes C^*(G))$ is as in (\ref{eq-wG}).
Then $W_B$ corresponds via (\ref{eq-isom}) to  the unitary $1\otimes W\in U\M(A_\mu\otimes \mathcal K\otimes C^*(G))$
of (\ref{eq-tildedelta}), and therefore we see from Proposition \ref{prop-coact-Fischer} that the coaction 
$\widetilde{\delta}_\mu=\Ad(1\otimes W)\circ (\delta_\mu\otimes_*\id_{\K})$ of $G$ on 
$A_\mu\otimes \K$ as in (\ref{eq-tildedelta}) corresponds to the dual coaction $\widehat{\beta}_\mu$  
on $B\rtimes_\mu G$, and hence
\begin{equation}\label{eq-tildebeta}
\widetilde{\beta}_\mu:=\Ad(W_B^*)\circ \widehat{\beta}_\mu
\end{equation}
corresponds to $\delta_\mu\otimes_*\id_{\K}$ on $A_\mu\otimes\K$. This implies that if we identify $A_\mu$ with 
$C(B\rtimes_\mu G, i_\K)$ as above, then $\delta_\mu$ can be recovered by the restriction of $\tilde{\beta}$ 
to $C(B\rtimes_\mu G, i_\K)$ as in Proposition \ref{prop-coact-Fischer}. 

Alternatively, we could have used Fischer's approach to the maximal crossed product 
$B\rtimes_\beta G=A\rtimes_\delta\hat G\rtimes_{\hatdelta} G$ 
to obtain the maximalization $(A_\max,\delta_\max)$ of $(A,\delta)$ and then passed to the appropriate quotient as in Proposition \ref{prop-max-to-mu}.

\section{A new approach to deformation}\label{sec-alternative}


We now introduce a new method for deformation by group twists, inspired by Fischer's version of Landstad duality. This extends the constructions of 
Bhowmick, Neshveyev, and Sangha (\cite{BNS}) for deformation by Borel cocycles in the reduced case. 
Note that in the previous sections we needed to consider $\mu$-crossed products for a single $C^*$-algebra only, whereas below,  we need to apply the $\mu$-crossed product  to a variety of $C^*$-algebras. Therefore, we assume from now on that $\rtimes_\mu$ is a duality crossed-product functor, such as $\rtimes_\max$ or $\rtimes_r$.

As in \cite{BE:deformation}, instead of using cocycles, we employ twists $\sigma$, since this avoids many awkward computations. Recall that the twist $\sigma=(\T\stackrel{\iota}\into  G_\sigma \stackrel{q}\onto G)$ 
is just a central extension $G_\sigma$ of $G$ by $\T$. In what follows, we shall often write $\tig$ or $\tilde{t}$ for elements in $G_\sigma$, and we write $g$, resp.\ $t$, for their images in $G=G_\sigma/\T$ under the quotient map.

Given a twist $\sigma$ as above, we obtain a Green twisted action \cite{Green} $(\id, \iota^\sigma)$ of the pair $(G_\sigma, \T)$ on the complex numbers $\C$, where we 
write $\iota^\sigma$ for the inclusion $\iota^\sigma:\T=\U(\C)\into \C$. The {\em twisted group algebra}
$C^*(G, \sigma)$ for the twist $\sigma$ is just the twisted crossed product $\C\rtimes_{(\id, \iota^\sigma)}G$ (see \cite{CELY}*{Chapter 1} for a survey on Green's twisted crossed products).

More generally, given an action $\beta:G\car B$ of the group $G=G_\sigma/\T$ on a $C^*$-algebra $B$, we obtain a twisted action $(\beta,\iota^\sigma)$ of $(G_\sigma,\T)$ on $B$ by inflating $\beta$ to $G_\sigma$ via the quotient map $q:G_\sigma\to G$ (which, by abuse of notation, we still call $\beta$) and by composing 
$\iota^\sigma:\T\to \C$ with the inclusion $\C\to U\M(B), \lambda \mapsto \lambda 1_B$ (which we still call $\iota^\sigma$). 
Following \cite{BE:deformation}, we define the space 
\begin{equation}\label{eq-C0Giota)}
C_0(G_\sigma, \bar\iota):=\{f\in C_0(G_\sigma): f(\tig z)=\bar{z}f(\tig)\quad \forall \tig\in G_\sigma, z\in\T\}
\footnote{in \cite{BE:deformation} we called this $C_0(G_\sigma,\iota)$.}
\end{equation}
equipped with the right translation action $\rt^\sigma:G_\sigma\car C_0(G_\sigma,\bar\iota)$ provides a $\rt-(\rt,\iota^\sigma)$ equivariant 
Morita equivalence, where left and right actions and inner products are given by suitable pointwise multiplication of functions. 

Given a weak $G\rtimes G$-algebra $(B,\beta,\phi)$, the diagonal action $\gamma:=\rt^\sigma\otimes \beta: G_\sigma\car C_0(G_\sigma,\bar\iota)\otimes_BB=:\E_\sigma(B)$
turns $\E_\sigma(B)$ into a full $(\beta,\iota^\sigma)$-equivariant Hilbert $B$-module such that the action 
$\beta^\sigma:=\Ad \gamma$ factors through an action of $G$ on $B^\sigma:=\K(\E_\sigma(B))$. 
Therefore  $(\E_\sigma(B), \gamma)$ becomes 
an equivariant  $(B^\sigma,\beta^\sigma)-(B,(\beta,\iota^\sigma))$ equivalence bimodule (identifying 
the action $\beta^\sigma:G\car B^\sigma$ with the inflated twisted action $(\beta^\sigma\circ q, 1_\T):(G_\sigma, \T)\car B^\sigma$ as 
described in \cite{Echterhoff:Morita_twisted}).
Together with the morphism $\phi^\sigma:C_0(G)\to \M(B_\sigma)=\L_B(\E_\sigma(B))$
induced by the left action of $C_0(G)$ on $C_0(G_\sigma,\bar\iota)\otimes_BB$, we obtain the
$\sigma$-deformed weak $G\rtimes G$-algebra $(B^\sigma, \beta^\sigma, \phi^\sigma)$.

Starting with a $\mu$-coaction $(A_\mu,\delta_\mu)$ for some duality crossed-product functor $\rtimes_\mu$ and 
the corresponding weak $G\rtimes G$-algebra $(B,\beta,\phi)=(A_\mu\rtimes_{\delta_\mu}\hatG, \hatdelta_\mu, j_{C_0(G)})$, the application of $\mu$-Landstad duality (either using the approach of \cite{BE:deformation} or the Fischer approach described above) to the deformed weak $G\rtimes G$-algebra $(B^\sigma,\beta^\sigma, \phi^\sigma)$ then yields the deformed $\mu$-coaction 
$(A_\mu^\sigma, \delta_\mu^\sigma)$.

\begin{notation}\label{note-deform}
     We call $(A_\mu^\sigma,\delta_\mu^\sigma)$ the {\em $\sigma$-deformation} of $(A_\mu,\delta_\mu)$.
  \end{notation}
 
We now want to introduce a more direct approach to deformation by using Fischer's methods
directly to the twisted crossed-product $B\rtimes_{(\beta,\iota^\sigma)}G$ together with the dual coaction $\widehat{(\beta,\iota^\sigma)}$ and an inclusion of compact operators as explained below. 
The $\rt-\beta$ equivariant morphism $\phi:C_0(G)\to B$ is also equivariant 
for the twisted actions $(\rt,\iota^\sigma)$ and $(\beta,\iota^\sigma)$, respectively. It therefore 
descents to give a $*$-homomorphism 
\begin{equation}\label{eq-Phisigma}
   \Phi^\sigma:= \phi\rtimes G: C_0(G)\rtimes_{(\rt,\iota^\sigma)}G\to \M(B\rtimes_{(\beta,\iota^\sigma)}G),
\end{equation}
and, similarly, if we replace  $B\rtimes_{(\beta,\iota^\sigma)}G$ by any exotic crossed product 
$B\rtimes_{(\beta,\iota^\sigma),\mu}G$. Now we have the following fact

\begin{lemma}\label{lem-twisted-compact}
    Let $L^2(G_\sigma, \iota)$ denote the subspace of $L^2(G_\sigma)$ consisting of all elements
    $\xi\in L^2(G_\sigma)$ which satisfy $\xi(\tig z)=z\xi(\tig)$ for all $\tig\in G_\sigma, z\in \T$.
    Then the pair $(M^\sigma, \rho^\sigma)$ given by
    $$(M^\sigma(f)\xi)(\tig)=f(g)\xi(\tig)\quad\text{and}\quad 
    (\rho^\sigma_{\tig}\xi)(\tilde{t})=\Delta(g)^{1/2}\xi(\tilde{t}\tig)
    $$
    is a covariant representation for the twisted action $(\rt,\iota^\sigma):(G_\sigma,\T)\car C_0(G)$ whose integrated form induces an isomorphism
    $$M^\sigma\rtimes\rho^\sigma: C_0(G)\rtimes_{(\rt,\iota^\sigma)}G\congto \K(L^2(G_\sigma,\iota)).$$
\end{lemma}
\begin{proof}
    The proof is a consequence of Green's version of the Mackey machine (e.g., see \cite{Green} or \cite{CELY}*{Chapter 1}) which implies that the only irreducible representation of $C_0(G)\rtimes_{(\rt,\iota^\sigma)}G$ is the one induced from the representation of $C_0(G)$ given by evaluation at $e\in G$. 
    One then checks that this representation is just the one described in the lemma.
\end{proof}

\begin{remark} We note that any chosen Borel section $\mathfrak s:G\to G_\sigma$ induces an isomorphism  
$L^2(G_\sigma,\iota)\congto L^2(G); \xi\mapsto \xi\circ \mathfrak s$.
\end{remark}

For the twisted action $(\beta,\iota^\sigma):(G_{\sigma},\T)\car B$ there is a dual coaction $$\widehat{(\beta, \iota^\sigma)}:B\rtimes_{(\beta,\iota^\sigma)}G\to \M(B\rtimes_{(\beta,\iota^\sigma)}G\otimes C^*(G))$$
given by the integrated form of the covariant homomorphism 
$(i_B\otimes 1, i_{G_\sigma}\otimes u)$ where $(i_B,i_{G_\sigma}):(B,G_\sigma)\to \M(B\rtimes_{(\beta,\iota^\sigma)}G)$ is the universal (twisted) representation of 
$(\beta,\iota^\sigma):(G_\sigma, \T)\car B$ and $u:G\to U\M(C^*(G))$ is the universal representation 
of $G$ (see \cite{Quigg-full}*{Proposition 3.1} where the construction is given 
for general Green-twisted crossed products).  
Similarly, we obtain the dual coaction $\widehat{(\rt,\iota^\sigma)}$ of $G$ 
on $C_0(G)\rtimes_{(\rt,\iota)}G\cong \K(L^2(G_\sigma,\iota))$ such that the 
$*$-homomorphism  $\Phi^\sigma$ of (\ref{eq-Phisigma})
is $\widehat{(\rt,\iota^\sigma)}-\widehat{(\beta,\iota^\sigma)}$ equivariant.
Recall that  $w_G:=(g\mapsto u_g)\in U\M(C_0(G)\otimes C^*(G))$.

\begin{lemma}\label{lem-onecocycle}
Let $W^\sigma=M^\sigma\otimes\id_G(w_G)\in U\M(\K(L^2(G_\sigma, \iota))\otimes C^*(G))$.
Then, the isomorphism $M^\sigma\rtimes \rho^\sigma: C_0(G)\rtimes_{(\rt,\iota)}G\to \K(L^2(G_{\sigma}, \iota))$ identifies the dual coaction 
$\widehat{(\rt,\iota^\sigma)}$ with the coaction 
$k\mapsto (W^\sigma)^*(k\otimes 1)W^\sigma$ of $G$ on $\K(L^2(G_\sigma,\iota))$.

As a consequence, if $(B,\beta,\phi)$ is a weak $G\rtimes G$-algebra, then  
$$W_B:=\big((i_B\circ \phi)\otimes \id_G\big)(w_G)\in U\M\big(B\rtimes_{(\beta,\iota^\sigma)}G\otimes C^*(G)\big)$$ is a one-cocycle for the dual coaction $\widehat{(\beta,\iota^\sigma)}$ on $B\rtimes_{(\beta,\iota)}G$
 such that 
the coaction $\epsilon:=\Ad W_B\circ \widehat{(\beta,\iota^\sigma)}$ fixes the image $\phi\rtimes G\big(C_0(G)\rtimes_{(\rt,\iota^\sigma)}G)\big)$ in $\M(B\rtimes_{(\beta,\iota^\sigma)}G)$.
\end{lemma}
\begin{proof}
It follows directly from the definition that, identifying $C_0(G)\rtimes_{(\rt,\iota^\sigma)}G$ with $\K(L^2(G_\sigma,\iota))$ as above, the dual coaction 
$\widehat{(\rt,\iota^\sigma)}$  on $\K(L^2(G_\sigma, \iota))$ is determined by the formulas
$$\widehat{(\rt,\iota^\sigma)}(M^\sigma(f))= M^\sigma(f)\otimes 1\quad\text{and}\quad \widehat{(\rt,\iota^\sigma)}(\rho^\sigma_{\tig})=\rho^\sigma_{\tig}\otimes u_g$$
for $f\in C_0(G)$ and $\tig\in G_\sigma$. 
Thus, to prove the lemma, we need to check the equations
$$M^\sigma(f)\otimes 1=(W^\sigma)^*(M^\sigma(f)\otimes 1)W^\sigma\quad\text{and} \quad\rho^\om_{\tig}\otimes u_g=(W^\sigma)^*(\rho^\sigma_{\tig}\otimes 1)W^\sigma$$
for all $f\in C_0(G)$ and all $\tig\in G_\sigma$.
The left equation is trivial since $W^\sigma$ commutes with $M^\sigma(f)\otimes 1$ for all $f\in C_0(G)$. For the right equation 
we identify $\M(\K(L^2(G_\sigma,\iota))\otimes C^*(G))$ with $\L( L^2(G_\sigma,\iota)\otimes C^*(G))$, the adjointable operators on the $C^*(G)$-Hilbert module
$L^2(G_\sigma,\iota)\otimes C^*(G)$, and then compute for any element
 $\xi\in L^2(G_\sigma,\iota)\otimes C^*(G)$ (viewed as a function $\xi:G_\sigma\to C^*(G)$):
\begin{align*}
\big((W^\sigma)^*(\rho^\sigma_{\tig}\otimes 1)W^\sigma\xi\big)(\tilde{t})&
=u_t^*\big(\rho^\sigma_{\tig}\otimes 1)W^\sigma\xi\big)(\tilde{t})\\
&=\sqrt{\Delta(g)}u_t^*(W^\sigma\xi)(\tilde{t}\tig)\\
&=\sqrt{\Delta(g)}u_t^*u_{tg}\xi(\tilde{t}\tig)\\
&=\sqrt{\Delta(g)}u_{g}\xi(\tilde{t}\tig)\\
&=\big((\rho^\sigma_{\tig}\otimes u_g)\xi\big)(\tilde{t}).
\end{align*}
The result follows.
The last statement is now a direct consequence of  the $\widehat{(\rt,\iota^\sigma)}-\widehat{(\beta, \iota^\sigma)}$ equivariance of 
$\Phi^\sigma:C_0(G)\rtimes_{(\rt,\iota^\sigma)}G\to \M(B\rtimes_{(\beta,\iota^\sigma)}G)$.
\end{proof}

Now, if we identify $C_0(G)\rtimes_{(\rt,\iota^\sigma)}G$ with $\K(L^2(G_\sigma,\iota))$, Fischer's methods as explained in \S \ref{subsec-Fischer} imply a decomposition of
the crossed  product $B\rtimes_{(\beta,\iota^\sigma)}G$ as a tensor product
$D_{\max}^\sigma\otimes \K(L^2(G_\sigma,\iota))$ with 
\begin{equation}
\label{eq-Amax}
D_{\max}^\sigma := 
\left\{ 
m \in \M(B \rtimes_{(\beta,\iota^\sigma)} G) \mid  
\begin{array}{l}
    \Phi^\sigma(k) m = m \Phi^\sigma(k) \in B \rtimes_{(\beta,\iota^\sigma)} G, \\
    \quad\mbox{for all}\quad k \in \K(L^2(G_\sigma,\iota))
\end{array}
\right\}
\end{equation}
such that the coaction 
$\epsilon=\Ad W_B\circ \widehat{(\beta,\iota^\sigma)}$ restricts to a well-defined coaction, say $\epsilon^\sigma_{\max}$, of $G$ on $D^\sigma_\max$. By the same reasoning, if we replace 
$B\rtimes_{(\beta,\iota^\sigma)}G$ with an exotic version 
    $B\rtimes_{(\beta,\iota^\sigma),\mu}G$ for some duality crossed-product functor $\rtimes_\mu$,
    we obtain a $\mu$-coaction $(D^\sigma_\mu,\epsilon^\sigma_\mu)$ by applying 
    Proposition \ref{prop-max-to-mu} to the maximal coaction $(D^\sigma_\max,\epsilon^\sigma_\max)$.\footnote{A priori, a crossed-product functor $\rtimes_\mu$ is not defined for twisted actions. But for a duality crossed-product functor $\rtimes_\mu$ we can define 
    $\big(B\rtimes_{(\beta,\iota^\sigma),\mu}G, \widehat{(\beta,\iota^\sigma)}_\mu\big)$
    as the $\mu$-ization of the maximal coaction 
     $\big(B\rtimes_{(\beta,\iota^\sigma)}G, \widehat{(\beta,\iota^\sigma)}\big)$ as 
     in Proposition \ref{prop-max-to-mu}. }

\begin{notation}\label{note-Fischerdeformation}
Starting above with $(B,\beta,\phi)=(A_\mu\rtimes_{\delta_\mu}\hatG, \hatdelta_\mu, j_{C_0(G)})$ for some $\mu$-coaction $(A_\mu,\delta_\mu)$, we call 
$(D^\sigma_\mu,\delta^\sigma_\mu)$ the {\em Fischer deformation} of $(A_\mu,\delta_\mu)$.
\end{notation}

\begin{remark}\label{rem-BNS}
If $\om\in Z^2(G,\T)$ is a Borel cocycle, let $\sigma_\om:=(\T\into G_\om\onto G)$ 
denote the associated central extension in which $G_\om=G\times \T$ (as a Borel space) 
equipped with the multiplication $(g, z)(t, w)=(gt, \om(g,t)zw)$.
The associated Green-twisted crossed products $C_0(G)\rtimes_{(\rt,\iota^{\sigma_\om})}G$
and $B\rtimes_{(\beta, \iota^{\sigma_\om})}G$ are 
then isomorphic to the more measure theoretic Busby-Smith crossed 
products $C_0(G)\rtimes_{\rt,\om}G$ and $B\rtimes_{\beta,\om}G$, respectively,
as used by Bhowmick, Neshveyev, and Sangha in \cite{BNS}.

To see the connection, let $\sigma=(\T\into G_\sigma\onto G)$ be any twist for $G$ and let us choose a Borel section $\mathfrak{s}:G\to G_\sigma$ for the quotient map $G_\sigma\stackrel{q}{\onto}G$. Then 
$$\om:G\times G\to \T; \om(g,t)=\mathfrak{s}(g)\mathfrak{s}(t)\mathfrak{s}(gt)^{-1}$$
is a corresponding cocycle  whose class $[\om]\in H^2(G,\T)$ classifies $\sigma$
(starting with $\om$ as above, we can recover $\om$ from $\sigma_\om$ via the section $\mathfrak{s}(g)=(g,1)\in G_{\om}$).
Following the construction of Green's twisted crossed product
as given in \cite{Green}*{p.~197}, we obtain
$B\rtimes_{(\beta,\iota^\sigma)}G$
as a completion of the convolution algebra
\begin{equation}\label{eq-Green-twisted-cp}
C_c(G_\sigma,B, \iota^\sigma)=\{f:G_\sigma\to B: f(z \tig )=f(\tig)\bar{z}\} 
\end{equation}
with convolution and involution given by the formulas
\begin{equation}\label{Green-twisted-conv}
f*_\sigma h(\tig)=\int_G f(\tilde{t})\beta_t(h(\tilde{t}^{-1}\tig))\, dt
\quad\text{and}\quad f^*(\tig)=\Delta(g^{-1})\beta_g(f(\tig^{-1}))^*
\end{equation}
 On the other hand, the Busby-Smith twisted crossed product $B\rtimes_{\beta,\om}G$ is a completion of the convolution algebra 
 $L^1(G, B,\om)$, that is $L^1(G,B)$ with convolution and involution given by
 \begin{equation}\label{eq-Busby-Smith}
     f*_\om h(g)=\int_G f(t)\beta_t(h(t^{-1}g))\om(t, t^{-1}g)\,dt \;\;\text{and}\;\;
     f^*(g)=\Delta(g^{-1})\overline{\om(g, g^{-1})}f(g^{-1})^*.
 \end{equation}
 It is then straightforward to check that $\Phi:C_c(G_\sigma, B,\iota^\sigma)\to L^1(G,B,\om); \Phi(f)= f\circ \mathfrak{s}$ extends to a $*$-isomorphism 
 $B\rtimes_{(\beta,\iota^\sigma)}G\cong B\rtimes_{\beta,\om}G$ which is equivariant for the dual coactions (and similarly for $C_0(G)\rtimes_{(\rt,\iota^\sigma)}G\cong C_0(G)\rtimes_{\rt,\om}G$) and  intertwines the inclusions $C_0(G)\rtimes_{(\rt,\iota^\sigma)}G\cong C_0(G)\rtimes_{\rt,\om}G$, respectively.

Hence, by Remark \ref{rem-Fischer-functorial},  all constructions above can be done as well
in terms of Busby-Smith crossed products and their dual coactions.
This yields deformed coactions
$(D^\om_\mu,\delta^\om_\mu)$ which, by  Remark \ref{rem-Fischer-functorial}, are isomorphic to the deformed coactions $(D^{\sigma_\om}_\mu,\epsilon^{\sigma_\om}_\mu)$
for the twist $\sigma_\om$ as above.

Starting then with the weak $G\rtimes G$-algebra 
$(B,\beta,\phi)=(A\rtimes_\delta \hatG, \hatdelta, j_{C_0(G)})$ for a {\em normal} coaction $(A,\delta)$ of $G$, it follows from the proof of 
\cite{BNS}*{Theorem 3.4} that the deformed algebra $A^\om$ as in \cite{BNS}
coincides with the commutator algebra $C(E,\iota(\K(L^2(G))))$ of Fischer's construction
with $E=B\rtimes_{\beta,\om, r}G$, the reduced twisted crossed product, and inclusion $\iota=\phi\rtimes_\om G:\K(L^2(G))\cong C_0(G)\rtimes_{\rt,\om}G\to \M(E)$ given via the canonical map, if we identify $B\rtimes_{\beta,\om, r}G$ with 
$\theta(A\rtimes_\delta \hatG\rtimes_{\hatdelta, \om, r}G)$ as in the proof of 
\cite{BNS}*{Theorem 3.4}. Thus $A^\om$ coincides with the Fischer deformed algebra
$D^{\om}_r$ with respect to the reduced crossed product as introduced above.  
We leave it to the reader to check that the 
coaction $\epsilon^\om_r$ on $D^\om_r$ also coincides with the coaction 
$\delta^\om$ on $A^\om$ as constructed in \cite{BNS}*{Theorem 4.1}.
\end{remark}

\section{Comparison of the deformation procedures}
In this section we want to show that for any $\mu$-coaction $(A_\mu,\delta_\mu)$ for a duality 
crossed product functor $\rtimes_\mu$ and for any twist $\sigma=(\T\into G_\sigma\onto G)$
the deformed cosystem $(A^\sigma_\mu,\delta^\sigma_\mu)$ of Notation \ref{note-deform}
coincides (up to isomorphism) with the Fischer deformation $(D^\sigma_\mu,\delta^\sigma_\mu)$ as in 
Notation \ref{note-Fischerdeformation}. In view of Proposition \ref{prop-max-to-mu} 
it suffices to show this for the maximal coactions $(A^\sigma_\max, \delta^\sigma_\max)$ and 
$(D^\sigma_\max,\epsilon^\sigma_\max)$. For the sake of brevity, we shall omit below the subscript ``$\max$'' and assume from now on that all our coactions (and crossed products) are maximal.

In order to prove the isomorphism $(A^\sigma, \delta^\sigma)\cong (D^\sigma, \epsilon^\sigma)$, we shall show that for any weak $G\rtimes G$-algebra $(B,\beta,\phi)$ 
there are isomorphisms 
$$C_0(G)\rtimes_{(\rt,\iota^\sigma)}G\cong C_0(G)\rtimes_\rt G\quad\text{and}\quad 
B\rtimes_{(\beta,\iota^\sigma)}G\cong B^\sigma\rtimes_{\beta^\sigma}G$$
which are equivariant for the respective dual coactions and intertwine the  inclusion
$\phi\rtimes G:C_0(G)\rtimes_{(\rt,\iota^\sigma)}G\to \M(B\rtimes_{(\beta,\iota^\sigma)}G)$ 
with the inclusion $\phi^\sigma\rtimes G:C_0(G)\rtimes_{\rt}G\to \M(B^\sigma\rtimes_{\beta^\sigma}G)$.
The result will then follow from the functoriality of Fischer's construction
(see Remark \ref{rem-Fischer-functorial}).

For the isomorphisms we shall use the following general observation, which makes use of the 
linking algebra $L(\X)=\left(\begin{smallmatrix}A&\X\\ \X^*& B\end{smallmatrix}\right)$ of an $A-B$ equivalence bimodule $_A\X_B$ together with the multiplier bimodule
$_{\M(A)}\M(\X)_{\M(B)}$ as studied in detail in \cite{ER} or \cite{EKQR}.
Recall, in particular, the equation 
$$\M(L(\X))=L(\M(\X))=\left(\begin{matrix} \M(A)&\M(\X)\\ \M(\X)^*& \M(B)\end{matrix}\right).$$
Recall further that if $\gamma:G_\sigma\car\X$ is an action of $G_\sigma$ on $\X$ which implements
an $(\alpha,\tau)-(\beta,\nu)$ equivariant Morita equivalence for twisted actions 
$(\alpha,\tau):(G_\sigma,\T)\car A$ and $(\beta,\nu):(G_\sigma, \T)\car B$, then 
they induce the twisted action
$$\left(\left(\begin{smallmatrix}\alpha&\gamma\\ \gamma^*&\beta\end{smallmatrix}\right),
\left(\begin{smallmatrix} \tau&0\\ 0&\nu\end{smallmatrix}\right)\right):(G_\sigma,\T)\car L(\X)$$
with Green-twisted crossed product
$$L(\X)\rtimes G:=L(\X)\rtimes_{\left(\left(\begin{smallmatrix}\alpha&\gamma\\ \gamma^*&\beta\end{smallmatrix}\right),
\left(\begin{smallmatrix} \tau&0\\ 0&\nu\end{smallmatrix}\right)\right)}G.$$ 
Taking corners with respect to the images $p,q\in \M(L(\X)\rtimes G)$ of the opposite full projections 
$p=\left(\begin{smallmatrix} 1&0\\0&0\end{smallmatrix}\right)_{L(\X)}$ and
$q=\left(\begin{smallmatrix} 0&0\\0&1\end{smallmatrix}\right)_{L(\X)}$, we see that 
$\X\rtimes_\gamma G:= p(L(\X)\rtimes_\gamma G)q$ becomes an imprimitivity bimodule 
for $A\rtimes_{(\alpha,\tau)}G\cong p(L(\X)\rtimes_\gamma G)p$ and 
$B\rtimes_{(\beta,\nu)}G=q(L(\X)\rtimes_\gamma G)q$. In particular, we obtain an identification
$$L(\X)\rtimes G\cong L(\X\rtimes_\gamma G).$$
Observe also, that the dual coaction of $G$ on $L(\X)\rtimes G$ compresses to 
the dual coactions $\widehat{(\alpha,\tau)}$, $\widehat{\gamma}$, and $\widehat{(\beta,\nu)}$ 
on the corners $A\rtimes_{(\alpha,\tau)}G$, $\X\rtimes_\gamma G$, and $B\rtimes_{(\beta,\nu)}G$, respectively, 
making  $(\X\rtimes_\gamma G,\widehat\gamma)$ a 
$(A\rtimes_{(\alpha,\tau)}G, \widehat{(\alpha,\tau)})-(B\rtimes_{(\beta,\nu)}G, \widehat{(\beta,\nu)})$
Morita equivalence, as studied in detail in \cite{EKQR}.

As a key towards the construction of our desired isomorphism, we shall use the following

 \begin{lemma}\label{lem-single}
Let $\X$ be an $A-B$ equivalence bimodule and suppose that $S\in \M(\X)$ such that $S^*S=1_{\M(B)}$ and 
$SS^*=1_{\M(A)}$. Then  $B\cong SBS^*=A$  via $b\mapsto SbS^*$. 
Here all multiplications are inside the linking algebra $L(\M(\X))$. 

If, in addition, $\delta_\X:\X\to \M(\X\otimes C^*(G))$ is a coaction of $G$ on $\X$ which implements a
Morita equivalence between the coactions $(A,\delta_A)$ and $(B,\delta_B)$, and such that 
$$\delta_\X(S)=S\otimes 1$$
(using the unique extension of $\delta_\X$ to $\M(\X)$)
then the above isomorphism $\Ad S: B\congto A$ is $\delta_B-\delta_A$ equivariant.
\end{lemma}
\begin{proof}
    The first assertion is straightforward, so we restrict to the second. So assume that 
    $\delta_\X(S)=S\otimes 1$. We then get for all $b\in B$:
    \begin{align*}
        \delta_A(SbS^*)&=\delta_\X(S)\delta_B(b)\delta_\X(S^*)\\
        &=(S\otimes 1)\delta_B(b)(S^*\otimes 1)\\
        &=\Ad S\otimes \id(\delta_B(b)),
    \end{align*}
    which is what we want.
    
\end{proof}

We want to apply this lemma first to the 
$C_0(G)\rtimes_{\rt}G-C_0(G)\rtimes_{(\rt,\iota^\sigma)}G$ equivalence bimodule 
$C_0(G_\sigma,\bar\iota)\rtimes_{\rt^\sigma}G$, the crossed product of the $\rt-(\rt,\iota^\sigma)$ equivariant 
$C_0(G)-C_0(G)$ equivalence $\rt^\sigma:G_\sigma\car C_0(G_\sigma,\bar\iota)$
as introduced in (\ref{eq-C0Giota)}) above.  We shall see below that this module admits 
an isomorphic representation as compact operators between Hilbert spaces.
For this recall from \cite{ER} that an imprimitivity-bimodule representation of an $A-B$-equivalence bimodule $\X$ on a pair of Hilbert spaces $(\H,\H')$ is a triple of linear maps
$$(\pi_A,\pi_\X,\pi_B):(A,\X, B)\to \big(\B(\H), \B(\H', \H), \B(\H')\big)$$
such that $\pi_A:A\to \B(\H), \pi_B:B\to \B(\H')$ are $*$-homomorphisms, and $\pi_\X:\X\to \B(\H',\H)$
is compatible with the canonical $\B(\H)-\B(\H')$ Hilbert-bimodule structure on $\B(\H',\H)$.
It is observed in \cite{ER}*{\S 2, Remarks (2)} that faithfulness of any of the maps in the triple 
$(\pi_A,\pi_\X,\pi_B)$ implies faithfulness of all the others. 
Notice that every imprimitivity-bimodule representation as above induces the representation 
$\left(\begin{smallmatrix}\pi_A&\pi_{\X}\\ \pi_\X^*&\pi_B\end{smallmatrix}\right)$ of the linking algebra $L(\X)=\left(\begin{smallmatrix} A& \X\\ \X^* & B\end{smallmatrix}\right)$ 
acting via matrix multiplication on $\left\{\left(\begin{smallmatrix} \xi\\ \eta\end{smallmatrix}\right):\xi\in \H,\eta\in \H'\right\}\cong \H\oplus\H'$.

Recall now that we have faithful representations 
$M\rtimes\rho:C_0(G)\rtimes_{\rt}G\to \K(L^2(G))$ and 
$M^\sigma\rtimes\rho^\sigma: C_0(G)\rtimes_{(\rt,\iota^\sigma)}G\to \K(L^2(G_\sigma,\iota))$.
Define $L^\sigma:C_0(G_\sigma,\bar\iota)\to \B(L^2(G_\sigma, \iota), L^2(G))$ by
\begin{equation}\label{eq-Lsigma}
L^\sigma(f)\xi(g)=f(\tig)\xi(\tig)\quad\forall f\in C_0(G_\sigma, \bar\iota),\xi\in L^2(G_\sigma,\iota),
\tig\in G_{\sigma}.
\end{equation}
It is then straightforward to check that $(M,L^\sigma,M^\sigma)$ is an imprimitivity bimodule 
representation of $_{C_0(G)}C_0(G_\sigma,\bar\iota)_{C_0(G)}$ on the pair of Hilbert spaces $(L^2(G), L^2(G_\sigma,\iota))$ such that the pair $\left(\left(\begin{smallmatrix} M& L^\sigma\\ (L^\sigma)^*& M^\sigma\end{smallmatrix}\right), \left(\begin{smallmatrix} \rho&0\\0&\rho^\sigma\end{smallmatrix}\right)\right)$ 
becomes a covariant representation for the twisted action 
$$(\Rt, \tau):=\left(\left(\begin{smallmatrix} \rt &\rt^\sigma\\ (\rt^\sigma)^* &\rt\end{smallmatrix}\right),
\left(\begin{smallmatrix} 1_\T&0\\0& \iota^\sigma\end{smallmatrix}\right)\right):(G_\sigma,\T)\car L(C_0(G_\sigma,\bar\iota)).$$

The representation therefore integrates to a \Star-representation, say $\Phi_L$, 
of the twisted crossed product 
$L(C_0(G_\sigma,\bar\iota))\rtimes_{(\Rt,\tau)} G\cong L\big(C_0(G_\sigma,\bar\iota)\rtimes_{\rt^\sigma}G\big)$
into $\B\big(L^2(G)\oplus L^2(G_\sigma,\bar\iota)\big)$. 
Since its compression to the upper left full corner $C_0(G)\rtimes_\rt G\cong\K(L^2(G))$ is irreducible,
it follows that the image of $\Phi_L$ is just the compact operators on $L^2(G)\oplus L^2(G_\sigma,\iota)$.
Compression of 
this representation to the upper left, upper right, and lower right corners 
 then yields the desired faithful imprimitivity bimodule representation 
$$(M\rtimes\rho, L^\sigma\rtimes\rho^\sigma, M^\sigma\rtimes \rho^\sigma)$$
of  $_{C_0(G)\rtimes_\rt G}\big(C_0(G_\sigma,\bar\iota)\rtimes_{\rt^\sigma}G\big)_{C_0(G)\rtimes_{(\rt,\iota^\sigma)}G}$ on   $\big(L^2(G), L^2(G_\sigma,\iota)\big)$
such that 
\begin{equation}\label{eq-imp-iso}
  L^\sigma\rtimes\rho^\sigma:   C_0(G_\sigma,\bar\iota)\rtimes_{\rt^\sigma}G\congto \K(L^2(G_\sigma, \iota), L^2(G)).
\end{equation}

Summarizing the above, we now get the following. 

\begin{proposition}\label{prop-imp-iso}
The representation $(M\rtimes\rho, L^\sigma\rtimes\rho^\sigma, M^\sigma\rtimes \rho^\sigma)$ identifies 
the $C_0(G)\rtimes_{\rt}G-C_0(G)\rtimes_{(\rt,\iota)}G$ imprimitivity bimodule $C_0(G_\sigma, \bar\iota)\rtimes_{\rt^\sigma}G$ with the $\K(L^2(G))-\K(L^2(G_\sigma,\iota))$ imprimitivity bimodule 
$\K(L^2(G), L^2(G_\sigma,\iota))$ and, therefore,  it extends to an isomorphism of multiplier bimodules
\begin{equation}\label{eq-mult-iso}
\begin{split}_{\M(C_0(G)\rtimes_\rt G)}\M(C_0(G_\sigma,\iota)\rtimes_{\rt^\sigma}&G)_{\M(C_0(G)\rtimes_{(\rt,\iota^\sigma)}G)}\\
&\cong_{\B(L^2(G))}\B(L^2(G_\sigma,\iota), L^2(G))_{\B(L^2(G_\sigma, \iota))}
\end{split}
\end{equation}
\end{proposition}

We now let $ U:L^2(G_\sigma,\iota)\to L^2(G)$ be any unitary isomorphism. Then its preimage $S\in \M(C_0(G_\sigma,\bar\iota)\rtimes_{\rt^\sigma}G)$
 satisfies the requirements of Lemma \ref{lem-single} above, and therefore induces an isomorphism 
$C_0(G)\rtimes_{\rt}G\cong C_0(G)\rtimes_{(\rt,\iota^\sigma)}G$. We want to choose $S$ in a way that makes this isomorphism 
equivariant with respect to the dual coactions. For this let us choose any Borel function $\varphi:G_\sigma\to \T$ which satisfies 
\begin{equation}\label{eq-varphi}
\varphi(\tig z)=\bar{z}\varphi(\tig)\quad\forall \tig \in G_\sigma, z\in \T.
\end{equation}
Note that any Borel section $\mathfrak s:G\to G_\sigma$ allows the construction of such function $\varphi$ by putting 
$$\varphi(\tig)=\bar{z}\quad\text{iff}\quad \tig =\mathfrak s(g)z.$$
Now define 
$$ L_\varphi: L^2(G_\sigma,\iota)\to L^2(G);\quad L_\varphi (\xi)=\varphi\cdot\xi,$$
and let us denote by $S_\varphi\in \M(C_0(G_\sigma,\bar\iota)\rtimes_{\rt^\sigma}G)$ its inverse image under the isomorphism 
(\ref{eq-mult-iso}). 

Recall from Lemma \ref{lem-onecocycle} that the isomorphisms $M\rtimes \rho: C_0(G)\rtimes_\rt G\congto \K(L^2(G))$ and 
 $M^\sigma\rtimes \rho^\sigma: C_0(G)\rtimes_{(\rt,\iota^\sigma)} G\congto \K(L^2(G_\sigma,\iota))$ transform the dual coactions
 to the coactions
 \begin{equation}\label{eq-dual-coaction-K}
     \begin{split}
\delta_{\K(L^2(G))}:\M(\K(L^2(G))\to C^*(G)); \;\;& k\mapsto W^*(k\otimes 1)W, \quad\text{and}\\
     \delta_{\K(L^2(G_\sigma,\iota))}:\M(\K(L^2(G_\sigma,\iota))\to C^*(G));\;\; &   k\mapsto (W^\sigma)^*(k\otimes 1)W^\sigma
     \end{split}
 \end{equation}
 for the unitaries $W=(M\otimes \id)(w_g)\in U\M(\K(L^2(G))\otimes C^*(G))$ and 
 $W^\sigma=(M^\sigma\otimes \id)(w_G)\in U\M(\K(L^2(G_\sigma,\iota))\otimes C^*(G))$, respectively.
Similarly, using the restriction of the representation $\left(\left(\begin{smallmatrix} M& L^\sigma\\ (L^\sigma)^*& M^\sigma\end{smallmatrix}\right)\rtimes \left(\begin{smallmatrix} \rho&0\\0&\rho^\sigma\end{smallmatrix}\right)\right)$ to $C_0(G_\sigma,\bar\iota)\rtimes_{\rt^\sigma}G$, a
similar computation as in the proof of Lemma \ref{lem-onecocycle} shows
that the dual coaction on $C_0(G_\sigma,\bar\iota)\rtimes_{\rt^\sigma}G$ transforms to the coaction 
$$k\mapsto W^*(k\otimes 1)W^\sigma$$ on $\K(L^2(G_\sigma,\iota), L^2(G))$.

\begin{lemma}\label{lem-equiv-iso}
Let $L_\varphi: L^2(G_\sigma,\iota)\to L^2(G)$ and  $S_\varphi\in \M(C_0(G)\rtimes_{(\rt,\iota^\sigma)}G)$ be as above. Then
\begin{equation}\label{eq-W-imp}
    W^*(L_\varphi\otimes 1) W^\sigma= L_\varphi\otimes 1.
\end{equation}
As a consequence, the preimage $S_\varphi\in \M(C_0(G_\sigma,\bar\iota)\rtimes_{\rt^\sigma}G)$ of $L_\varphi$ under the isomorphism $\M(C_0(G_\sigma,\bar\iota)\rtimes_{\rt^\sigma}G)\cong \B(L^2(G_\sigma,\iota), L^2(G))$ is fixed by the dual coaction $\widehat{\rt^\sigma}$ and the
isomorphism $\Ad S_\varphi: C_0(G)\rtimes_{(\rt,\iota^\sigma)}G\congto C_0(G)\rtimes_\rt G$ of Lemma \ref{lem-single} is equivariant
for the dual coactions.
\end{lemma}
\begin{proof}
   For  $\xi\otimes \psi\in L^2(G_\sigma,\iota)\otimes C_c(G) \subseteq L^2(G_\sigma,\iota)\otimes C^*(G)$ we compute
   \begin{align*}
       \big(W^*(L_\varphi\otimes 1) W^\sigma(\xi\otimes x)\big)(g, t)&= \big((L_\varphi\otimes 1) W^\sigma(\xi\otimes x)\big)(g, gt)\\
       &=\varphi(\tig)\big(W^\sigma(\xi\otimes x)\big)(\tig, gt)\\
       &=\varphi(\tig)(\xi\otimes x)(\tig, t)\\
       &=\big((L_\varphi\otimes 1)(\xi\otimes x)\big)(g, t)
   \end{align*}
   This proves the equation (\ref{eq-W-imp}). In particular, it follows that the isomorphism 
   $\Ad L_\varphi: \K(L^2(G_\sigma,\iota))\congto \K(L^2(G))$ is equivariant for the respective coactions 
   as in (\ref{eq-dual-coaction-K}). Since $S_\varphi$ is the inverse image of $L_\varphi$ under the isomorphism 
   (\ref{eq-mult-iso}), the result follows.
\end{proof}

The above lemma will now easily implement a similar result for the $B^\sigma\rtimes_{\beta^\sigma}G-B\rtimes_{(\beta,\iota^\sigma)}G$ equivalence bimodule
$\E_\sigma(B)\rtimes_{\gamma}G$ and, similarly, for their exotic counter parts.
To prepare for this, we first observe that we have a $\rt^\sigma-\gamma$ equivariant linear map
$\psi: C_0(G_\sigma,\bar\iota)\to \M(\E_\sigma(B))\cong \Lb_B(B,\E_\sigma(B))$ given by
$$\psi(f)b:=f\otimes b\in C_0(G_\sigma,\bar\iota)\otimes_{C_0(G)}B=\E_\sigma(B).$$
The triple $(\phi^\sigma, \psi,\phi)$ then becomes a nondegenerate
$(G_\sigma,\T)$-equivariant imprimitivity bimodule map
$$_{C_0(G)}C_0(G_\sigma,\bar\iota)_{C_0(G)}\to _{\M(B^\sigma)}\M(\E_\sigma(B))_{\M(B)}$$
with a corresponding nondegenerate $*$-homomorphism 
$$\left(\begin{smallmatrix} \phi^\sigma& \psi\\ \psi^*&\phi\end{smallmatrix}\right):L(C_0(G_\sigma,\bar\iota))\to L(\M(\E_\sigma(B)))=\M(L(\E_\sigma(B))).$$
It then descents to a nondegenerate $*$-homomorphism
$$\left(\begin{smallmatrix} \phi^\sigma& \psi\\ \psi^*&\phi\end{smallmatrix}\right)\rtimes G:
L(C_0(G_\sigma,\bar\iota)\rtimes_{\rt^\sigma}G)\to \M(L(\E_\sigma(B)\rtimes_\gamma G))$$
mapping corners to corners and therefore decomposing to a matrix of maps
$$\left(\begin{smallmatrix} \phi^\sigma& \psi\\ \psi^*&\phi\end{smallmatrix}\right)\rtimes G
=:\left(\begin{smallmatrix} \phi^\sigma\rtimes G& \psi\rtimes G\\ (\psi\rtimes G)^*&\phi\rtimes G\end{smallmatrix}\right).$$
By nondegeneracy, it extends to  
$$\M(L(C_0(G_\sigma,\bar\iota)\rtimes_{\rt^\sigma}G))=L(\M(C_0(G_\sigma,\bar\iota)\rtimes_{\rt^\sigma}G))).$$
Since every descent of an equivariant $*$-homomorphism to the crossed products is equivariant for the dual coactions, we see that the map
$$\psi\rtimes G: \M(C_0(G_\sigma,\bar\iota)\rtimes_{\rt^\sigma}G)\to \M(\E_\sigma(B)\rtimes_\gamma G)$$
sends the element $S_\varphi\in \M(C_0(G_\sigma,\bar\iota)\rtimes_{\rt^\sigma}G)$ of 
Lemma \ref{lem-equiv-iso} to an element, say $R_\varphi\in \M(\E_\sigma(B)\rtimes_\gamma G)$ which satisfies 
all the requirements of Lemma \ref{lem-single}: 
we have 
 $$   R_\varphi^*R_\varphi=\phi(S_\varphi^*S_\varphi)=\phi(1_{C_0(G)})=1_B $$
and similarly $R_\varphi R_\varphi^*=1_{B^\sigma}$. Moreover, we have 
\begin{align*}
\widehat{\gamma}(R_\varphi)&=\widehat{\gamma}(\psi(S_\varphi))
=(\psi\otimes \id)\big(\widehat{\rt^\sigma}(S_\varphi)\big)\\
&=(\psi\otimes \id)\big(S_\varphi\otimes 1\big)
=R_\varphi\otimes 1
\end{align*}
Thus, applying Lemma \ref{lem-single} we now get

\begin{proposition}\label{prop-deformation-iso}
Let $(B,\beta,\phi)$ be a weak $G\rtimes G$-algebra and let $\sigma=(\T\into G_\sigma\onto G)$ be a twist for $G$. Then the element $R_\varphi\in \M(\E_\sigma(B)\rtimes_\gamma G)$ constructed above induces a $\widehat{(\beta,\iota^\sigma)}-\widehat{\beta^\sigma}$ equivariant $*$-isomorphism 
$$\Ad R_\varphi: B\rtimes_{(\beta,\iota^\sigma)}G\to B^\sigma\rtimes_{\beta^\sigma}G$$
such that following diagram commutes:
$$
\begin{CD}
    C_0(G)\rtimes_{(\rt,\iota^\sigma)}G @>\phi\rtimes G>> \M(B\rtimes_{(\beta,\iota^\sigma)}G)\\
    @V\Ad S_\varphi VV    @VV\Ad R_\varphi V\\
    C_0(G)\rtimes_\rt G  @>>\phi^\sigma\rtimes G> \M(B^\sigma\rtimes_{\beta^\sigma}G)
\end{CD}$$
\end{proposition}
\begin{proof}
    Everything, except (maybe) the commutativity of the diagram,  follows directly from the discussion preceding the proposition. 
    But the commutativity of the diagram follows from the equation
    $R_\varphi=\psi\rtimes G(S_\varphi)$ and the fact that the triple
    $(\phi^\sigma\rtimes G,\psi\rtimes G,\phi\rtimes G)$ is an imprimitivity bimodule map. This leads to the computation
    \begin{align*}
        \phi^\sigma\rtimes G(S_\varphi  x S_\varphi^*)&=
        \psi\rtimes G(S_\varphi)(\phi\rtimes G(x))\psi\rtimes G(S_\varphi)^*\\
        &=R_\varphi(\phi\rtimes G(x))R_\varphi^*. 
    \end{align*}
    
\end{proof}

As a direct consequence of Proposition \ref{prop-deformation-iso} we can now finally conclude

\begin{theorem}\label{thm-deformation-iso}
    Let $(A,\delta)$ be a maximal coaction and let 
    $\sigma=(\T\into G_\sigma\onto G)$ be a twist for $G$. Then the maximal 
    deformation $(A^\sigma, \delta^\sigma)$ and the maximal 
    Fischer deformation $(D^\sigma, \epsilon^\sigma)$ are equivariantly isomorphic. 

    As a consequence (using Proposition \ref{prop-max-to-mu}) the same holds true for the $\mu$-deformations $(A_\mu^\sigma, \delta_\mu^\sigma)$ and 
    $(D_\mu^\sigma, \epsilon_\mu^\sigma)$  for any duality crossed-product functor $\rtimes_\mu$.
\end{theorem}
\begin{proof}
    Just apply Proposition \ref{prop-deformation-iso} to the weak $G\rtimes G$-algebra $(B,\beta,\phi)=(A\rtimes_\delta \hatG, \hatdelta, j_{C_0(G)})$ and 
    use the functoriallity of Fischer's construction.
\end{proof}

One might wonder, whether the isomorphism between $B^\sigma\rtimes_{\beta^\sigma}G$ and $B\rtimes_{(\beta,\iota^\sigma)}G$ of Proposition \ref{prop-deformation-iso} has a more direct description. 
This is indeed the case if the function $\varphi:G_\sigma\to \T$ in the construction 
of the operators $L_\varphi, S_\varphi$ and $R_\varphi$, respectively, can be chosen to be continuous (which is equivalent to the existence of a continuous section $\mathfrak s \colon G\to G_\sigma$ for the quotient map). 
In this case the element $\varphi$ can be regarded as an element of $\M(C_0(G_\sigma,\bar\iota))= \Lb_{C_0(G)}(C_0(G), C_0(G_\sigma,\bar\iota))$ 
given by $f\mapsto \varphi\cdot f$. Its image in 
$\M(\E_\sigma(B))=\Lb_B(B, \E_\sigma(B))$ 
is given by $b\mapsto \varphi\otimes b$ (writing $b=\phi(f)b'$ for some $f\in C_0(G), b'\in B$, we see that 
$\varphi\otimes b=\varphi\cdot f\otimes b'\in C_0(G_\sigma,\bar\iota)\otimes_{C_0(G)}B=\E_\sigma(B)$). As a result 
we get an identification $B\cong \varphi\otimes B\cong \E_\sigma(B)$ as Hilbert $B$-module. For the action $\gamma:G_\sigma\car \E_\sigma(B)$ we compute
\begin{align*}
  \gamma_{\tig}(\varphi\otimes b)&=\rt^\sigma_{\tig}(\varphi)\otimes \beta_g(b)
  = \varphi\otimes \phi(\bar{\varphi}\cdot\rt^\sigma_{\tig})(\varphi)b\\
  &= \varphi\otimes \phi(u(\tig))\beta_g(b)
\end{align*}
with $u(\tig)=\bar\varphi\cdot\rt^\sigma_{\tig}(\varphi)\in C(G,\T)=U\M(C_0(G))$.
Thus, identifying $B$ with $\E_\sigma(B)$ as above, the action 
$\gamma$ is given by $\gamma_{\tig}(b)=\phi(u(\tig))\beta_g(b)$. It induces the 
action $\beta^\sigma=\Ad\gamma_g=\Ad \phi(u(\tig))\circ \beta_g$ on 
$B=\K_B(B)\cong\K_B(\E_\sigma(B))=B^\sigma$. 
Indeed, one can check that $\phi\circ u:G_\sigma \to U\M(B)$ is a $(\beta, \iota^\sigma)$ one-cocycle which induces an exterior equivalence between the twisted actions $(\beta,\iota^\sigma)$ and $(\beta^\sigma, 1_\T)$  of $(G_\sigma, \T)$ (see \cite{Echterhoff:Morita_twisted}*{p.~175} for a definition).
Therefore, we obtain the isomorphism 
$$\Phi:B\rtimes_{\beta^\sigma}G\congto B\rtimes_{(\beta,\iota^\sigma)}G$$
that extends the map
$$\Phi:C_c(G, B)\to C_c(G_\sigma, B,\iota^\sigma);\quad \Phi(f)(\tig)=f(g)\phi(u(\tig))^*. $$
This is indeed the isomorphism of Proposition \ref{prop-deformation-iso} in this case.

In general, if $\varphi$ can only be chosen to be  Borel, the isomorphism of Proposition \ref{prop-deformation-iso} can be interpreted as a substitute of a suitable multiplication of functions $f\in C_c(G_\sigma, B, \iota^\sigma)$ with 
the Borel-function 
$(\tig,t)\mapsto u(\tig)(t)=\overline{\varphi(\tilde{t})}\varphi(\tilde{t}\tig))$. It is difficult to give this a precise meaning in a direct way if $\varphi$ is not continuous.

\medskip
To compare the above with earlier constructions for continious cocycles, assume again that $\varphi:G_\sigma\to\T$ can be chosen continuous. It then induces a continuous section $\mathfrak{s}:G\to G_\sigma$ by
\begin{equation}\label{eq-tig}
    \mathfrak{s}(g)=\tig:\Leftrightarrow g=q(\tig)\;\text{and}\; \varphi(\tig)=1.
\end{equation} 
We  obtain the associated {\em continuous} cocycle $\om \in Z^2_{\mathrm{cont}}(G,\T)$ by $\om(g,t)=\partial \mathfrak{s}(g,t)=\mathfrak{s}(\tig)\mathfrak{s}(\tilde{t})\mathfrak{s}(\tig\tilde{t})^{-1}$.
Composing the isomorphism $\Phi$ above with the isomorphism
$$\Psi:B\rtimes_{(\beta,\iota^\sigma)}G\congto B\rtimes_{\beta,\om}G;\; f\mapsto f\circ \mathfrak{s} \quad(\text{for}\; f\in C_c(G_\sigma, B, \iota^\sigma) )$$
of Remark \ref{rem-BNS} we obtain the isomorphism 
\begin{equation}\label{eq-isom-om}
\Psi\circ \Phi: B\rtimes_{\beta^\sigma}G\congto B\rtimes_{\beta,\om} G; \;f\mapsto f\phi(u(\mathfrak{s}(g)))^*\quad\text{(for $f\in C_c(G,B)$)}.
\end{equation}  Using the equation 
$\varphi(\tilde{t})=\tilde{t}^{-1}\mathfrak{s}(t)\in \T$ for $\tilde{t}\in G_\sigma$, which follows from (\ref{eq-tig}), we then compute
\begin{align*}
    u(\mathfrak{s}(g))(t)&=
    \overline{\varphi(\tilde{t})}\varphi(\tilde{t}\mathfrak{s}(g))
    =\mathfrak{s}(t)^{-1}\tilde{t}\mathfrak{s}(g)^{-1}\tilde{t}^{-1}\mathfrak{s}(tg)\\
    &\stackrel{(*)}{=}\mathfrak{s}(g)^{-1}\mathfrak{s}(t)^{-1}\mathfrak{s}(tg)
    =\overline{\om(t, g)}
\end{align*}
where for equation $(*)$ we conjugated the central element $\mathfrak{s}(t)^{-1}\tilde{t}$ by $\mathfrak{s}(g)^{-1}$.
We thus recover the exterior equivalence between $(\beta,\iota^\om)$ and $\beta^\om$ as described in \cite{BE:deformation}*{Remark 3.4} and the isomorphism (\ref{eq-isom-om}) above is given on $C_c(G,B)$ by a suitable multiplication with the 
function $(g,t)\mapsto\overline{\om(t,g)}$.

\section{Outlook and future work}

We believe that the methods developed in this paper, particularly those grounded in Fischer’s framework for Landstad duality and maximalizations, are robust enough to extend beyond the setting of locally compact groups. In particular, they are well suited for generalization to regular locally compact quantum groups. Since the key structural ingredients, such as equivariant Hilbert modules, coactions trivial on compact operators, and the duality framework, are available in the quantum group setting (cf. Fischer \cite{Fischer-PhD}), we anticipate that analogous deformation constructions can be formulated for coactions of quantum groups, including twisted and exotic versions, thus extending constructions of Neshveyev and Tuset (\cite{NT}) in the reduced case. We plan to pursue a detailed treatment of this extension  in future work.

\end{document}